\pdfoutput=1

\documentclass[a4paper,reqno, 12pt,titlepage]{amsart} 
\usepackage[czech, english]{babel}
\usepackage[utf8]{inputenc}
\usepackage{microtype}
\usepackage{amssymb, graphicx, enumerate, enumitem, color}
\usepackage[usenames,dvipsnames,svgnames,table]{xcolor}
\usepackage[foot]{amsaddr}
\usepackage{thmtools}

\colorlet{mylinkcolor}{violet}
\colorlet{mycitecolor}{YellowOrange}
\colorlet{myurlcolor}{Aquamarine}
\usepackage[unicode=true]{hyperref}
\hypersetup{
    colorlinks,
    linkcolor={red!50!black},
    citecolor={blue!50!black},
    urlcolor={blue!80!black}
}

\usepackage[ruled,titlenumbered,linesnumbered,noend,norelsize]{algorithm2e}
\DontPrintSemicolon
\SetNlSty{footnotesize}{}{}
\IncMargin{0.5em}
\SetNlSkip{1em}
\SetKwIF{If}{ElseIf}{Else}{if}{then}{else if}{else}{endif}
\SetKwFor{Loop}{Loop}{}{EndLoop}
\SetKw{Return}{return}
\SetNlSkip{2em}
\SetKwProg{Proc}{Procedure}{}{}

\newenvironment{algorithm-hbox}{\hbadness=10000\begin{algorithm}}{\end{algorithm}}

\newtheorem{theorem}{Theorem}
\newtheorem{conjecture}[theorem]{Conjecture}
\newtheorem{corollary}[theorem]{Corollary}

\newtheorem{claim}[theorem]{Claim}
\newtheorem{obs}[theorem]{Observation}
\newtheorem{lemma}[theorem]{Lemma}
\theoremstyle{remark}

\newcommand{\set}[1]{\{#1\}}
\newcommand{\norm}[1]{{\left|#1\right|}}

\newcommand{\calC}{\mathcal{C}}

\newcommand{\N}{\mathbb{N}}
\newcommand{\R}{\mathbb{R}}

\newcommand{\WD}{\operatorname{WD}}
\newcommand{\WU}{\operatorname{WU}}
\newcommand{\wup}{\operatorname{wu}}
\newcommand{\wdown}{\operatorname{wd}}

\newcommand{\w}{\operatorname{w}}
\newcommand{\leftmost}{\operatorname{left}}

\DeclareMathOperator\Inc{Inc}

\DeclareMathOperator\val{value}

\DeclareMathOperator\Up{U}
\DeclareMathOperator\D{D}

\DeclareMathOperator\wcol{wcol}
\DeclareMathOperator\WR{WReach}

\let\le\leqslant

\let\leq\leqslant
\let\geq\geqslant

\let\subset\subseteq

\let\epsilon\varepsilon

\renewenvironment{enumerate}{\begin{enumorig}[label=\textup{(\roman*)}, noitemsep, topsep=2pt plus 2pt, labelindent=.2em, leftmargin=*, widest=iii]}{\end{enumorig}}

\newenvironment{enumerate2}{\begin{enumorig}[label=\textup{(2.\arabic*)}, noitemsep, topsep=2pt plus 2pt, labelindent=.2em, leftmargin=*, widest=iii]}{\end{enumorig}}

\renewenvironment{itemize}{\begin{itemorig}[label=\textbullet, noitemsep, topsep=2pt plus 2pt, labelindent=.5em, labelsep=.5em, leftmargin=*]}{\end{itemorig}}

\makeatletter
\let\old@setaddresses\@setaddresses
\def\@setaddresses{\bigskip\bgroup\parindent 0pt\let\scshape\relax\old@setaddresses\egroup}
\makeatother

\begin{document}
\title{Nowhere Dense Graph Classes and Dimension}

\author[G.~Joret]{Gwena\"{e}l Joret}
\address[G.~Joret]{Computer Science Department \\
  Universit\'e Libre de Bruxelles, 
  Brussels, 
  Belgium}
\email{gjoret@ulb.ac.be}

\author[P.~Micek]{Piotr Micek}
\address[P.~Micek]{Theoretical Computer Science Department\\
  Faculty of Mathematics and Computer Science, Jagiellonian University, Krak\'ow, Poland 
--- and --- Institute of Mathematics, Combinatorics and Graph Theory Group \\
Freie Universit\"at Berlin, Berlin, Germany}
\email{piotr.micek@tcs.uj.edu.pl}

\author[P.~Ossona~de~Mendez]{Patrice Ossona~de~Mendez}
\address[P.~Ossona~de~Mendez]{
Centre d'Analyse et de Math\'ematiques Sociales (CNRS, UMR 8557)\\
  190-198 avenue de France, 75013 Paris, France
	--- and ---
Computer Science Institute of Charles University (IUUK)\\
   Malostransk\' e n\' am.\ 25, 11800 Praha 1, Czech Republic}
\email{pom@ehess.fr}

\author[V.~Wiechert]{Veit Wiechert}
\address[V.~Wiechert]{Institut f\"ur Mathematik\\
  Technische Universit\"at Berlin, 
  Berlin, 
  Germany}
\email{wiechert@math.tu-berlin.de}

\thanks{G.\ Joret is supported by an ARC grant from the Wallonia-Brussels Federation of Belgium. 
P.\ Micek is partially supported by a Polish National Science Center grant (SONATA BIS 5; UMO-2015/18/E/ST6/00299).  
G.\ Joret and P.\ Micek also acknowledge support from a joint grant funded by the Belgian National Fund for Scientific Research (F.R.S.--FNRS) and the Polish Academy of Sciences (PAN). 
V.\ Wiechert is supported by the Deutsche Forschungsgemeinschaft within the research training group `Methods for Discrete Structures' (GRK 1408). 
P.\ Ossona de Mendez is supported by grant ERCCZ LL-1201 
and  by the European Associated Laboratory ``Structures in
Combinatorics'' (LEA STRUCO)}

\begin{abstract}
Nowhere dense graph classes provide one of the least restrictive notions of sparsity for graphs.
Several equivalent characterizations of nowhere dense classes have been obtained over the years, using a wide range of combinatorial objects. 
In this paper we establish a new characterization of nowhere dense classes, in terms of poset dimension: 
A monotone graph class is nowhere dense if and only if for every $h \geq 1$ and every $\epsilon > 0$, posets of height at most $h$ with $n$ elements and whose cover graphs are in the class have dimension $\mathcal{O}(n^{\epsilon})$. 
\end{abstract}

\maketitle

\section{Introduction}

A class of graphs is {\em nowhere dense} if for every $r \geq 1$, there exists $t\geq 1$ such that no graph in the class contains a subdivision of the complete graph $K_t$ where each edge is subdivided at most $r$ times as a subgraph. 
Examples of nowhere dense classes include most sparse graph classes studied in the literature, such as planar graphs, graphs with bounded treewidth, graphs excluding a fixed (topological) minor, graphs with bounded maximum degree, graphs that can be drawn in the plane with a bounded number of crossings per edges, and more generally graph classes with bounded expansion. 

At first sight, being nowhere dense might seem a weak requirement for a graph class to satisfy. 
Yet, this notion captures just enough structure to allow solving a wide range of algorithmic problems efficiently: In their landmark paper, Grohe, Kreutzer, and Siebertz~\cite{GKS17} proved for instance that every first-order property can be decided in almost linear time on graphs belonging to a fixed nowhere dense class. 

One reason nowhere dense classes attracted much attention in recent years is the realization that they can be characterized in several, seemingly different ways. 
Algorithmic applications in turn typically build on the `right' characterization for the problem at hand and sometimes rely on multiple ones, such as in the proof of Grohe {\it et al.}~\cite{GKS17}. 
Nowhere dense classes were characterized in terms of 
shallow minor densities~\cite{NOdM-nowhere-dense}  and consequently  in terms of generalized coloring numbers (by results from \cite{Zhu09}),
low tree-depth colorings~\cite{NOdM-nowhere-dense} (by results from \cite{NOdM-decomp}), and subgraph densities in shallow minors \cite{Taxi_hom}; they were  also  characterized in terms of quasi-uniform wideness~\cite{NOdM10, KRS17, PST17}, 
the so-called splitter game~\cite{GKS17}, 
sparse neighborhood covers~\cite{GKS17}, 
neighborhood complexity~\cite{EGKKPRS16}, 
the model theoretical notion of stability~\cite{AA14}, as well as existence of particular analytic limit objects~\cite{NodM-modelings}. 
The reader is referred to the survey on nowhere dense classes by Grohe, Kreutzer, and Siebertz~\cite{GKS13} for an overview of the different characterizations, and to the textbook by Ne{\v{s}}et{\v{r}}il and Ossona de Mendez~\cite{NOdM-book} for a more general overview of the various notions of sparsity for graphs  (see also \cite{SurveyND}). 

The main contribution of this paper is a new characterization of nowhere dense classes that brings together graph structure theory and the combinatorics of partially ordered sets (posets). 
Informally, we show that the property of being nowhere dense can be captured by looking at the dimension of posets whose order diagrams are in the class when seen as graphs.

Recall that the {\em dimension $\dim(P)$} of a poset $P$ is the least integer $d$ such that the elements of $P$ can be embedded into $\R^d$ in such a way that $x<y$ in $P$ if and only if the point of $x$ is below the point of $y$ with respect to the product order of $\R^d$. 
Dimension is a key measure of a poset's complexity.  

The standard way of representing a poset is to draw its \emph{diagram}: 
First, we draw each element as a point in the plane, in such a way that if $a<b$ in the poset then $a$ is drawn below $b$. 
Then, for each relation $a<b$ in the poset not implied by transitivity (these are called \emph{cover relations}), we draw a $y$-monotone curve going from $a$ up to $b$. 
The diagram implicitly defines a corresponding undirected graph, where edges correspond to pairs of elements in a cover relation. 
This is the \emph{cover graph} of the poset.
Let us also recall that the {\em height} of a poset is the maximum size of a chain in the poset (a set of pairwise comparable elements). 

Recall that a {\em monotone} class means a class closed under taking subgraphs. 
Our main result is the following theorem.  

\begin{theorem}
\label{thm:main}
Let $\calC$ be a monotone class of graphs. 
Then $\calC$ is nowhere dense if and only if for every integer $h\geq 1$ and real number $\epsilon>0$, $n$-element posets of height at most $h$ whose cover graphs are in $\calC$ have dimension $\mathcal{O}(n^\epsilon)$.
\end{theorem} 

This result is the latest step in a series of recent works connecting poset dimension with graph structure theory. 
This line of research began with the following result of Streib and Trotter~\cite{ST14}: For every fixed $h\geq 1$, posets of height $h$ with a planar cover graph have bounded dimension. 
That is, the dimension of posets with planar cover graphs is bounded from above by a function of their height. 
This is a remarkable theorem, because in general bounding the height of a poset does not bound its dimension, as shown for instance by the height-$2$ posets called {\em standard examples}, depicted in Figure~\ref{fig:kelly} (left). 
Requiring the cover graph to be planar does not guarantee any bound on the dimension either, as shown by Kelly's construction~\cite{Kel81} of posets with planar cover graphs containing large standard examples as induced subposets (Figure~\ref{fig:kelly}, right). 
Thus, it is the combination of the two ingredients, bounded height and planarity, that implies that the dimension is bounded. 

\begin{figure}[t]
\centering
\includegraphics[scale=1.0]{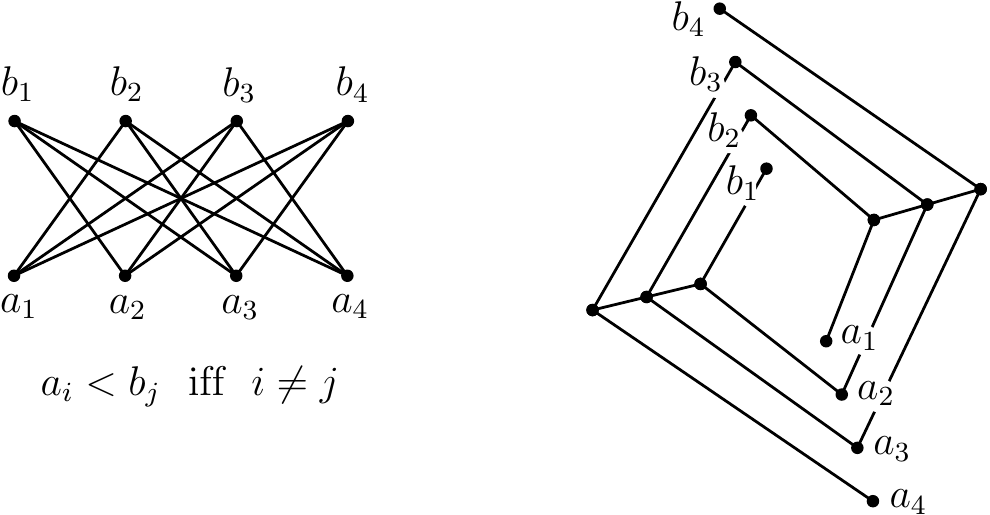}
\caption{\label{fig:kelly} The standard example $S_4$ (left) and Kelly's construction containing $S_4$ (right). 
The {\em standard example} $S_m$ ($m\geq 2$) is the height-$2$ poset consisting of $m$ minimal elements $a_1, \dots, a_m$ and $m$ maximal elements $b_1, \dots, b_m$ and the relations $a_i < b_j$ for all $i, j \in [m]$ with $i\neq j$. 
It has dimension $m$.}
\end{figure}

Soon afterwards, it was shown in a sequence of papers that requiring the cover graph to be planar in the Streib-Trotter result could be relaxed: 
Posets have dimension upper bounded by a function of their height if their cover graphs
\begin{itemize}
\item have bounded treewidth, bounded genus, or more generally exclude an apex-graph as minor~\cite{JMMTWW};
\item exclude a fixed graph as a (topological) minor~\cite{Walczak17, MW15}; 
\item belong to a fixed class with bounded expansion~\cite{JMW17+}. 
\end{itemize}

A class of graphs has {\em bounded expansion} if for every $r \geq 1$, there exists $c\geq 0$ such that no graph in the class contains a subdivision of a graph with average degree at least $c$ where each edge is subdivided at most $r$ times as a subgraph. 
This is a particular case of nowhere dense classes. 

Zhu~\cite{Zhu09} characterized bounded expansion classes as follows: A class has bounded expansion if and only if for every $r \geq 0$, there exists $c \geq 1$ such that every graph in the class has weak $r$-coloring number at most $c$. 
Weak coloring numbers were originally introduced by Kierstead and Yang~\cite{KY03} as a generalization of the degeneracy of a graph (also known as the \emph{coloring number}). 
They are defined as follows.  
Let $G$ be a graph and consider some linear order $\pi$ on its vertices (it will be convenient to see $\pi$ as ordering the vertices of $G$ from left to right).  
Given a path $Q$ in $G$, we denote by $\leftmost(Q)$ the leftmost vertex of $Q$ w.r.t.\ $\pi$.
Given a vertex $v$ in $G$ and an integer $r\geq 0$, we say that $u\in V(G)$ is \emph{weakly $r$-reachable from} $v$ w.r.t.\ $\pi$ if there exists a path $Q$ of length at most $r$ from $v$ to $u$ in $G$ such that $\leftmost(Q)=u$.
We let $\WR_r^\pi[v]$ denote the set of weakly $r$-reachable vertices from $v$ w.r.t.\  $\pi$ (note that this set contains $v$ for all $r\geq 0$).
The {\em weak $r$-coloring number} $\wcol_r(G)$ of $G$ is defined as
\[
 \wcol_r(G):=\min_{\pi} \max_{v\in V(G)} |\WR_r^\pi[v]|.
\]

The novelty of our approach in this paper is that we bound the dimension of a poset using weak coloring numbers of its cover graph. 
Indeed, the general message of the paper is that dimension works surprisingly well with weak coloring numbers. 
We give a first illustration of this principle with the following theorem: 

\begin{theorem}
\label{thm:dim-wcol}
Let $P$ be a poset of height at most $h$, let $G$ denote its cover graph, and let $c := \wcol_{3h-3}(G)$. Then 
 \[
  \dim(P)\leq 4^c.
 \]
\end{theorem}

To prove this, we first make the following observation about weak reachability.  
\begin{obs}\label{obs-weak-reachability} 
Let $G$ be a graph and let $\pi$ be a linear order on its vertices. 
If $w, x, y, z$ are vertices of $G$ such that $w$ is weakly $k$-reachable from $x$ (w.r.t.\ $\pi$), $y$ is weakly $\ell$-reachable from $z$, 
and $Q$ is a path from $x$ to $z$ in $G$ of length at most $m$ such that $\leftmost(\set{w,y})\leq_{\pi} \leftmost(Q)$, then
\begin{center}
\text{one of $w,y$ is weakly $(k+\ell+m)$-reachable from the other.}
\end{center}
In particular, this holds if $m=0$, i.e.\ $x=z$.
\end{obs}
\begin{proof}
Consider a path $Q^{(1)}$ from $w$ to $x$ witnessing that $w$ is weakly $k$-reachable from $x$, and a path $Q^{(2)}$ from $z$ to $y$ witnessing that $y$ is weakly $\ell$-reachable from $z$.
The union of $Q$, $Q^{(1)}$ and $Q^{(2)}$ contains a path $Q^{(3)}$ connecting $w$ to $y$ of length at most $k+\ell+m$.
Since $w$ is the leftmost vertex of $Q^{(1)}$ in $\pi$ and $y$ is the leftmost vertex of $Q^{(2)}$ in $\pi$, and $\leftmost(\set{w,y})\leq_{\pi} \leftmost(Q)$ we have that one of $w$ and $y$ is the leftmost vertex of $Q^{(3)}$ in $\pi$. 
This proves that one of $w,y$ is weakly $(k+\ell+m)$-reachable from the other.
\end{proof}

Before continuing with the proof, let us introduce some necessary definitions regarding posets.  
Let $P$ be a poset.  
An element $y$ {\em covers} an element $x$ if $x < y$ in $P$ and there is no element $z$ such that $x < z < y$ in $P$. 
A chain $X$ of $P$ is said to be a {\em covering chain} if the elements of $X$ can be enumerated as $x_1, x_2, \dots, x_k$ in such a way that $x_{i+1}$ covers $x_i$ in $P$ for each $i\in [k-1]$. 
(We use the notation $[n] := \{1, \dots, n\}$.)  
The {\em upset $\Up[x]$} of an element $x\in P$ is the set of all elements $y\in P$ such that $x \leq y$ in $P$. 
Similarly, the {\em downset $\D[x]$} of an element $x\in P$ is the set of all elements $y\in P$ such that $y \leq x$ in $P$. 
Note that $x\in \Up[x]$ and $x\in \D[x]$. 
Given a subset $S$ of elements of $P$, we write $\Up[S]$ for the set $\bigcup_{x\in S}\Up[x]$, and define $\D[S]$ similarly. 

An {\em incomparable pair} of $P$ is an ordered pair $(x,y)$ of elements of $P$ that are incomparable in $P$.
We denote by $\Inc(P)$ the set of incomparable pairs of $P$.
Let $I \subseteq \Inc(P)$ be a non-empty set of incomparable pairs of $P$.
We say that $I$ is \emph{reversible} if there is a linear extension $L$ of $P$ \emph{reversing} each pair of $I$, that is, we have $x>y$ in $L$ for every $(x,y)\in I$. 
We denote by $\dim(I)$ the least integer $d$ such that $I$ can be partitioned into $d$ reversible sets. 
We will use the convention that $\dim(I)=1$ when $I$ is an empty set.   
As is well known, the dimension $\dim(P)$ of $P$ can equivalently be defined as $\dim(\Inc(P))$, that is, the least integer $d$ such that the set of all incomparable pairs of $P$ can be partitioned into $d$ reversible sets. 
This is the definition that we will use in the proofs. 

A sequence $(x_1,y_1), \dots, (x_k,y_k)$ of incomparable pairs of $P$ with $k \geq 2$ is said to be an \emph{alternating cycle of size $k$} if $x_i\leq y_{i+1}$ in $P$ for all $i\in\set{1,\ldots,k}$ (cyclically, so $x_k\le y_1$ in $P$ is required). 
(We remark that possibly $x_i=y_{i+1}$ for some $i$'s.) 
Observe that if $(x_1,y_1), \dots, (x_k,y_k)$ is a alternating cycle in $P$, then this set of incomparable pairs cannot be reversed by a linear extension $L$ of $P$. 
Indeed, otherwise we would have $y_i < x_i \leq y_{i+1}$ in $L$ for each $i \in \{1,2, \dots, k\}$ cyclically, which cannot hold. 
Hence, alternating cycles are not reversible. 
The converse is also true, as is well known: 
A set $I$ of incomparable pairs of a poset $P$ is reversible if and only if $I$ contains no alternating cycles.

We may now turn to the proof of Theorem~\ref{thm:dim-wcol}. 

\begin{proof}[Proof of Theorem~\ref{thm:dim-wcol}]
Let $\pi$ be a linear order on the elements of $P$ such that $\WR_{3h-3}^\pi[x]\leq c$ for each $x\in P$. 
Here and in the rest of the proof, weak reachability is to be interpreted w.r.t.\ the cover graph $G$ of $P$ and the ordering $\pi$. 

First, we greedily color the elements of $P$ using the ordering $\pi$ from left to right. 
When element $x$ is about to be colored, we give $x$ the smallest color $\phi(x)$ in $[c]$ that is not used for elements of $\WR_{3h-3}^\pi[x] -\{x\}$.
Let $x\in P$ and $y,z\in \WR_{h-1}^{\pi}[x]$.
By Observation~\ref{obs-weak-reachability}, either $y\in\WR_{2h-2}^{\pi}[z]$ or $z\in\WR_{2h-2}^{\pi}[y]$.
Therefore,
\begin{equation}
\label{eq:unique-color}
\text{$\phi(y)\neq \phi(z)$ for every $x\in P$ and $y,z\in \WR_{h-1}^{\pi}[x]$ with $y\neq z$.}
\end{equation}
In the proof, we will focus on elements $y$ in $\WR_{h-1}^\pi[x]$ that are weakly reachable from $x$ via covering chains that either start at $x$ and end in $y$, or the other way round.  
This leads us to introduce the \emph{weakly reachable upset $\WU[x]$} and the \emph{weakly reachable downset $\WD[x]$} of $x$:
 \begin{align*}
  \WU[x]:=\{y\in \Up[x] \exists\text{ covering chain $Q$ from $x$ to $y$ such that }\leftmost(Q)=y\},\\
  \WD[x]:=\{y\in \D[x]: \exists\text{ covering chain $Q$ from $y$ to $x$ such that }\leftmost(Q)=y\}.
 \end{align*}
Then $\WD[x], \WU[x] \subseteq \WR_{h-1}^{\pi}[x]$.
If $X$ is a set of elements of $P$, we write $\phi(X)$ for the set of colors $\{\phi(x) : x\in X\}$. 
Given an element $x$ of $P$ and a color $i\in\phi(\WU[x])$, 
by~\eqref{eq:unique-color} there is a unique element in $\WU[x]$ with color $i$; 
let us denote it by $\wup_i(x)$. 
Similarly, given $i\in\phi(\WD[x])$, we let $\wdown_i(x)$ denote the unique element in $\WD[x]$ with color $i$.  

For each $(x,y)\in\Inc(P)$, define the \emph{signature} $(A,B,C)$ of $(x,y)$, where
\[
A=\phi(\WU[x]),\ B=A\cap\phi(\WD[y]),\ C=\set{i\in B\mid \wup_i(x)<_{\pi} \wdown_i(y)}.
\]
As $[c]\supseteq A \supseteq B \supseteq C$, the number of possible signatures is at most $4^c$.
It remains to show that the set of incomparable pairs with a given signature  is reversible. 
This will show that $P$ has dimension at most $4^c$, as desired. 

Arguing by contradiction, suppose that there is a signature $(A,B,C)$ such that the set of incomparable pairs with signature $(A,B,C)$ is not reversible. 
Then these incomparable pairs contain an alternating cycle $(x_1,y_1),\ldots,(x_k,y_k)$. 

For each $j\in [k]$, consider all covering chains witnessing the comparability $x_j\leq y_{j+1}$ in $P$ (indices are taken cyclically) and choose one such covering chain $Q_j$ such that $q_j=\leftmost(Q_j)$ is as far to the left as possible w.r.t.\ $\pi$. 
Without loss of generality we may assume that $q_1$ is leftmost w.r.t.\ $\pi$ among the $q_j$'s. 

Let $t:=\phi(q_1)$.  
Clearly, $x_1\leq q_1\leq y_{2}$ in $P$, and $q_1\in\WU[x_1]\cap \WD[y_{2}]$. 
Thus, $t \in\phi(\WU[x_1])=A=\phi(\WU[x_{2}])$ and $t\in\phi(\WD[y_{2}])$, and hence $t\in \phi(\WU[x_{2}]) \cap \phi(\WD[y_{2}]) = B$.  
It follows that $\wup_t(x_j)$ and  $\wdown_t(y_j)$ are both defined for each $j\in [k]$. 

First suppose that $t\in C$. 
In particular, $\wup_t(x_2) <_{\pi} \wdown_t(y_2)$.
Thus
\[
\wup_t(x_2) <_{\pi} \wdown_t(y_2) = q_1 \leq_{\pi} q_2 = \leftmost(Q_2).
\]
Since $\wup_t(x_2)$ is $(h-1)$-weakly reachable from $x_2$, 
$\wdown_t(y_3)$ is $(h-1)$-weakly reachable from $y_3$, 
and $Q_2$ is a path in $G$ connecting $x_2$ and $y_3$ of length at most $h-1$ such that $\wup_t(x_2)<_{\pi}\leftmost(Q_2)$,
by Observation~\ref{obs-weak-reachability} one of $\wup_t(x_2)$, $\wdown_t(y_3)$ is weakly $(3h-3)$-reachable from the other. 
Since $\phi(\wup_t(x_2))=t=\phi(\wdown_t(y_3))$, we must have $\wdown_t(y_3) = \wup_t(x_2)=:q^*$ by~\eqref{eq:unique-color}. 
Hence, $x_2 \leq q^* \leq y_3$ in $P$, but $q^*<_{\pi} q_1\leq_\pi q_2$, which contradicts the way $q_2$ and $q_1$ were chosen. 

Next, suppose that $t\notin C$. 
Then, $\wdown_t(y_1) \leq_{\pi} \wup_t(x_1)$.
Note that $\wdown_t(y_1) \neq \wup_t(x_1)$, since otherwise we would have $x_1 \leq y_1$ in $P$. 
Thus, $\wdown_t(y_1) <_{\pi} \wup_t(x_1)$, and
\[
\wdown_t(y_1) <_{\pi} \wup_t(x_1)  = q_1 \leq_{\pi} q_k = \leftmost(Q_k).
\]
Since $\wdown_t(y_1)$ is $(h-1)$-weakly reachable from $y_1$, 
$\wup_t(x_k)$ is $(h-1)$-weakly reachable from $x_k$, 
and $Q_k$ is a path in $G$ connecting $x_k$ and $y_1$ of length at most $h-1$ such that $\wdown_t(y_1)<_{\pi}\leftmost(Q_k)$,
by Observation~\ref{obs-weak-reachability} one of $\wdown_t(y_1)$, $\wup_t(x_k)$ is weakly $(3h-3)$-reachable from the other. 
Since $\phi(\wup_t(x_k))=t=\phi(\wdown_t(y_1))$, we must have $\wup_t(x_k) = \wdown_t(y_1)=:q^*$ by~\eqref{eq:unique-color}. 
Hence, $x_k \leq q^* \leq y_1$ in $P$, but $q^*<_{\pi} q_1\leq_\pi q_k$, which contradicts the way $q_k$ and $q_1$ were chosen. 
\end{proof}

By Zhu's theorem, if we restrict ourselves to posets with cover graphs $G$ in a fixed class $\calC$ with bounded expansion, then $\wcol_{3h-3}(G)$ is bounded by a function of $h$. 
Thus Theorem~\ref{thm:dim-wcol} implies the theorem from~\cite{JMW17+} for classes with bounded expansion. 
However, the above proof is much simpler and implies better bounds on the dimension than those following from previous works (see the discussion in Section~\ref{sec:applications}). 
We see this as a first sign that weak coloring numbers are the right tool to use in this context. 

In~\cite{JMW17+}, it is conjectured that bounded expansion captures exactly situations where dimension is bounded by a function of the height: 

\begin{conjecture}[\cite{JMW17+}]
\label{conj:bounded_exp}
A monotone class of graphs $\calC$ has bounded expansion if and only if for every fixed $h\geq 1$, posets of height at most $h$ whose cover graphs are in $\calC$ have bounded dimension.
\end{conjecture}

While the result of~\cite{JMW17+} (reproved above) shows the forward direction of the conjecture, the backward direction remains surprisingly (and frustratingly) open.

By contrast, showing the backward direction of Theorem~\ref{thm:main} for nowhere dense classes is a straightforward matter, as we now explain. 
We prove the contrapositive. 
Thus let $\calC$ be a monotone graph class which is {\em not} nowhere dense (such a class is said to be {\em somewhere dense}). 
Our aim is to prove that there exist $h \geq 1$ and $\epsilon > 0$ such that there are $n$-element posets of height at most $h$ with dimension $\Omega(n^\epsilon)$ whose cover graphs are in $\calC$. 

Since $\calC$ is somewhere dense, there exists an integer $r\geq 0$ (depending on $\calC$) such that for every $t \geq 1$ there is a graph $G \in \calC$ containing an $\leq\!r$-subdivision of $K_t$ as a subgraph.  
(An {\em $\leq\!k$-subdivision} of a graph is a subdivision such that each edge is subdivided at most $k$ times.) 
Since $\calC$ is closed under taking subgraphs, this means that for every $m \geq 2$, 
the class $\calC$ contains a graph $G_m$ that is an $\leq\!r$-subdivision of the cover graph of the standard example $S_m$. 
Notice that $G_m$ has at most $rm^2+2m$ vertices.
Now it is easy to see that $G_m$ is also the cover graph of a poset $P_m$ of height at most $r+2$ containing $S_m$ as an induced subposet (simply perform the edge subdivisions on the diagram of $S_m$ in the obvious way). 
Let $n$ be the number of elements of $P_m$. 
The poset $P_m$ has dimension at least $m$, and thus its dimension is $\Omega(\sqrt{n})$ since $n \leq rm^2+2m$. 
Hence, we obtain the desired conclusion with $h := r+2$ and $\epsilon := 1/2$. 
This completes the proof of the backward direction of Theorem~\ref{thm:main}.

The non-trivial part of Theorem~\ref{thm:main} is that $n$-element posets of bounded height with cover graphs in a nowhere dense class have dimension $\mathcal{O}(n^\epsilon)$ for all $\epsilon > 0$.  
To prove this, we use the following characterization of nowhere dense classes in terms of weak coloring numbers~\cite{NOdM-nowhere-dense}:  
A class is nowhere dense if and only if for every $r \geq 0$ and every $\epsilon >0$, every $n$-vertex graph in the class has weak $r$-coloring number $\mathcal{O}(n^{\epsilon})$. 
(Note that this characterization does not require that the class is monotone, as deleting edges cannot increase the weak colouring numbers.)

We remark that it is a common feature of several characterizations in the literature that bounded expansion and nowhere dense classes can be characterized using the same graph invariants, but requiring $\mathcal{O}(1)$ and $\mathcal{O}(n^{\epsilon})$ $\forall \epsilon >0$ bounds on the invariants respectively. 
Thus, it is natural to conjecture the statement of Theorem~\ref{thm:main}, and indeed it appears as a conjecture in~\cite{JMW17+}.  
(We note that it was originally Dan Kr\'a{\v l} who suggested to the first author to try and show Theorem~\ref{thm:main} right after the result in~\cite{JMW17+} was obtained.)  

The $4^c$ bound in Theorem~\ref{thm:dim-wcol} unfortunately falls short of implying the forward direction of Theorem~\ref{thm:main}. 
Indeed, if the cover graph $G$ has $n$ vertices and belongs to a nowhere dense class, we only know that $\wcol_{3h-3}(G)\in \mathcal{O}(n^{\epsilon})$ for every $\epsilon > 0$. 
Thus from the theorem we only deduce that $\dim(P) \leq 4^{\mathcal{O}(n^{\epsilon})}$ for every $\epsilon > 0$, which is a vacuous statement since $\dim(P) \leq n$ always holds. 

In order to address this shortcoming, we developed a second upper bound on the dimension of a height-$h$ poset in terms of the weak $w(h)$-coloring number of its cover graph $G$ (for some function $w$) and another invariant of $G$. 
This extra invariant is the smallest integer $t$ such that $G$ does not contain an $\leq\!s(h)$-subdivision of $K_t$ as a subgraph, for some function $s$. 
The key aspect of our bound is that, for fixed $h$ and $t$, it depends polynomially on the weak $w(h)$-coloring number that is being considered.  
Its precise statement is as follows. 
(Let us remark that the particular values $w(h):=4h-4$ and $s(h):=2h-3$ used in the theorem are not important for our purposes, any functions $w$ and $s$ would have been good.)

\begin{theorem}
\label{thm:nowhere-dense-upper-bound} 
There exists a function $f:\N \times \N \to \N$ such that for every $h\geq 1$ and $t\geq 1$, every poset $P$ of height at most $h$ whose cover graph $G$ contains no $\leq\!(2h-3)$-subdivision of $K_t$ as a subgraph satisfies 
\[
\dim(P) \leq (3c)^{f(h,t)}, 
\]
where $c:=\wcol_{4h-4}(G)$.
\end{theorem}

Recall that for every nowhere dense graph class $\calC$ and every $r \geq 1$, there exists $t\geq 1$ such that no graph in $G$ contains an $\leq\!r$-subdivision of $K_t$ as a subgraph. 
Hence, Theorem~\ref{thm:nowhere-dense-upper-bound} implies the following corollary. 

\begin{corollary}
\label{cor:nowhere-dense-upper-bound}
For every nowhere dense class of graphs $\calC$, there exists a function $g:\N \to \N$ such that every poset $P$ of height at most $h$ whose cover graph $G$  is in $\calC$ satisfies
\[
\dim(P) \leq (3c)^{g(h)}, 
\]
where $c:=\wcol_{4h-4}(G)$. 
\end{corollary} 

For every integer $h\geq 1$ and real number $\epsilon >0$, this in turn gives a bound of $\mathcal{O}(n^\epsilon)$ on the dimension of $n$-element posets of height at most $h$ whose cover graphs $G$ are in $\calC$. 
Indeed, if we take $\epsilon':= \epsilon / g(h)$, then $\wcol_{4h-4}(G) \in \mathcal{O}(n^{\epsilon'})$ by the aforementioned characterization of nowhere dense classes~\cite{NOdM-nowhere-dense}, and hence $\dim(P)\in \mathcal{O}(n^{g(h)\epsilon'})=\mathcal{O}(n^{\epsilon})$ by the corollary. 
Therefore, this establishes the forward direction of Theorem~\ref{thm:main}.  

Let us also point out that Corollary~\ref{cor:nowhere-dense-upper-bound} provides another proof of the theorem from~\cite{JMW17+} for classes with bounded expansion, since $\wcol_{4h-4}(G)$ is bounded by a function of $h$ only when $\calC$ has bounded expansion. 
However, the proof is more involved than that of Theorem~\ref{thm:dim-wcol} and the resulting bound on the dimension is typically larger. 
Indeed, the bound in Theorem~\ref{thm:nowhere-dense-upper-bound} becomes interesting when the weak coloring number under consideration grows with the number of vertices. 

Our proof of Theorem~\ref{thm:nowhere-dense-upper-bound} takes its roots in the alternative proof due to Micek and Wiechert~\cite{MW15} of Walczak's theorem~\cite{Walczak17}, that bounded-height posets whose cover graphs exclude $K_t$ as a topological minor have bounded dimension. 
This proof is essentially an iterative algorithm which, if the dimension is large enough (as a function of the height), explicitly builds a subdivision of $K_t$, one branch vertex at a time. 
This is very similar in appearance to what we would like to show, namely that if the dimension is too big, then the cover graph contains a subdivision of $K_t$ where each edge is subdivided a bounded number of times (by a function of the height). 
The heart of our proof is a new technique based on weak coloring numbers, Lemma~\ref{lemma:q-support}, which we use to bound the number of subdivision vertices.

The paper is organized as follows. 
We prove Theorem~\ref{thm:nowhere-dense-upper-bound} in Section~\ref{sec:proof_nowhere_dense}. 
Next, we discuss in Section~\ref{sec:applications} improved bounds implied by Theorem~\ref{thm:dim-wcol} for special cases that were studied in the literature, such as for posets with planar cover graphs and posets with cover graphs of bounded treewidth. 
Finally, we close the paper in Section~\ref{sec:open_problems} with a couple open problems.

\section{Nowhere Dense Classes} 
\label{sec:proof_nowhere_dense}

As discussed in the introduction, the forward direction of Theorem~\ref{thm:main} follows from Theorem~\ref{thm:nowhere-dense-upper-bound} combined with Zhu's characterization of nowhere dense classes.  
In this section we prove Theorem~\ref{thm:nowhere-dense-upper-bound}. 
We begin with our key lemma. 

\begin{lemma}\label{lemma:q-support}
 Let $P$ be a poset of height $h$ with cover graph $G$, let $I \subseteq \Inc(P)$, and let $c:= \wcol_{4h-4}(G)$.  
 Then there exists an element $q\in P$ such that the set 
 $I':= \{(x,y)\in I: q \leq y \text{ in } P\}$ satisfies  
 $$
 \dim(I')\geq \dim(I) / c - 2.
 $$
\end{lemma}
\begin{proof}
Fix a linear order $\pi$ of the vertices of $G$ witnessing $\wcol_{4h-4}(G)\leq c$.
Here and in the rest of the proof, weak reachability is to be interpreted w.r.t.\ the cover graph $G$ and the ordering $\pi$. 

Let $\phi$ be a greedy vertex coloring of $G$ obtained by considering the vertices one by one according to $\pi$, and assigning to each vertex $z$ a color $\phi(z)\in [c]$ different from all the colors used on vertices in $\WR_{4h-4}^\pi[z]-\{z\}$.
Note that for every two vertices $x,y \in \WR_{2h-2}^\pi[z]$,  we know from Observation~\ref{obs-weak-reachability} that one of $x,y$ is weakly  $(4h-4)$-reachable from the other, and thus $\phi(x)\neq \phi(y)$.

For each $z\in P$, let $\tau(z)\in[c]$ denote the color of $\leftmost(\D[z])$. 
Given a color $i\in[c]$, let $\w_i(z)$ denote the unique element of $\WR_{2h-2}^\pi[z]$ colored $i$ if there is one, and leave $\w_i(z)$ undefined otherwise. 
Observe that $\w_{\tau(z)}(z) = \leftmost(\D[z])$.  
In particular, $\w_{\tau(z)}(z)\leq z$ in $P$.

Let $x,y\in P$ with $x\leq y$ in $P$. 
We claim that $\w_{\tau(y)}(x)=\w_{\tau(y)}(y)$. 
Indeed, $x$ and $y$ are at distance at most $h-1$ in $G$, and the element $\w_{\tau(y)}(x)$ is weakly $(2h-2)$-reachable from $x$ and $\w_{\tau(y)}(y)$ is weakly $(h-1)$-reachable from $y$. 
Hence, one of $\w_{\tau(y)}(x),\w_{\tau(y)}(y)$ is $(4h-4)$-reachable from the other by  Observation~\ref{obs-weak-reachability}, and $\w_{\tau(y)}(x)=\w_{\tau(y)}(y)$.

Define the \emph{signature} $\sigma(x,y)$ of a pair $(x,y) \in I$ to be the pair $(\tau(y),\alpha(x,y))$, where 
 \[
  \alpha(x,y):=\begin{cases}
               1 & \quad \text{if } \w_{\tau(y)}(x)=\w_{\tau(y)}(y)\\
               2 & \quad \text{if } \w_{\tau(y)}(x)<_\pi \w_{\tau(y)}(y)\\
               3 & \quad \text{if } \w_{\tau(y)}(x)>_\pi \w_{\tau(y)}(y) \text{ or } \w_{\tau(y)}(x) \text{ is not defined}.\\
              \end{cases}
 \]
 
 For each color $\tau\in[c]$ and value $\alpha\in[3]$, let $J_{\tau,\alpha}$ be the set of incomparable pairs $(x,y)\in I$ such that $\sigma(x,y)=(\tau,\alpha)$. Note that the sets $J_{\tau,\alpha}$ form a partition of $I$.
 \begin{claim}
  For each color $\tau\in[c]$, the sets $J_{\tau,2}$ and $J_{\tau,3}$ are reversible. 
 \end{claim}
 \begin{proof}
 Let $\alpha\in \{2,3\}$. 
 Arguing by contradiction, suppose that $J_{\tau,\alpha}$ is not reversible, and let $(x_1,y_1),\ldots,(x_k,y_k)$ denote an alternating cycle. 
  Since $x_1 \leq y_2$ in $P$, we have that $\w_{\tau(y_2)}(x_1) = \w_{\tau(y_2)}(y_2)$. 
  Since $\tau(y_2)=\tau=\tau(y_1)$, it follows that $\w_{\tau(y_1)}(x_1)$ is defined. 
 
 Since for every $i\in[k]$ we have $x_i\leq y_{i+1}$ in $P$ (cyclically), we obtain that $\w_{\tau}(x_i)=\w_{\tau}(y_{i+1})$.
 However, by our signature function this implies $\w_{\tau}(y_{i+1})= \w_{\tau}(x_i)<_\pi \w_{\tau}(y_i)$ for all $i\in[k]$ if $\alpha=2$, or $\w_{\tau}(y_{i+1})= \w_{\tau}(x_i)>_\pi \w_{\tau}(y_i)$ for all $i\in[k]$ if $\alpha=3$,  which  cannot hold cyclically.
 \end{proof}
 
 Since
 \[
 I = \bigcup_{\tau\in [c], \alpha \in [3]} J_{\tau,\alpha}
 \]
 the previous claims imply that
 \begin{align*}
  \dim(I) & \leq \sum_{\tau\in[c]}\dim(J_{\tau,1})+\sum_{\tau\in[c]}\dim(J_{\tau,2})+\sum_{\tau\in[c]}\dim(J_{\tau,3})\\
  & \leq \sum_{\tau\in[c]}\dim(J_{\tau,1}) + 2c.
 \end{align*}
 It follows that there exists a color $\tau\in[c]$ such that
 $ \dim(J_{\tau,1})\geq \dim(I)/c-2$.
 In the rest of the proof we focus on the set $J_{\tau,1}$.  Thus, denoting this set  by $I_\tau$, we have
  \begin{equation*}
  \dim(I_\tau)\geq \dim(I)/c-2.
 \end{equation*}

Given an element $p\in P$, we denote by $I_{\tau,p}$  the set of incomparable pairs $(x,y)\in I_\tau$ such that $p=\w_{\tau}(x)=\w_{\tau}(y)$. 
 Note that the sets $I_{\tau,p}$ ($p\in P$) partition $I_\tau$.

\begin{claim}\label{claim-max}
$\displaystyle{\dim(I_\tau)=\max_{p\in P} \ \dim(I_{\tau,p})}$.
\end{claim}
\begin{proof}
Let $d := \max_{p\in P} \ \dim(I_{\tau,p})$. 
Note that $\dim(I_\tau) \geq d$ since $I_{\tau,p} \subseteq I_\tau$ for each $p\in P$. 
Thus it remains to show that $\dim(I_\tau) \leq d$. 

For each $p\in P$, there exists a partition of $I_{\tau,p}$ into at most $d$ reversible sets. 
Let $I^1_{\tau,p}, \dots, I^d_{\tau,p}$ be disjoint reversible sets such that  
\[
I_{\tau,p}=\bigcup_{j\in [d]}I^j_{\tau,p}, 
\]
some sets being possibly empty.
We claim that the set $\bigcup_{p\in P}I^j_{\tau,p}$ is reversible for each $j\in[d]$.
Arguing by contradiction, suppose that for some $j\in[d]$ this set is not reversible.
Then it contains an alternating cycle $(x_1,y_1),\ldots,(x_k,y_k)$.
As $x_i\leq y_{i+1}$ in $P$ for $i\in[k]$, we have $\w_{\tau}(x_i)=\w_{\tau}(y_{i+1})$, which by the signatures of these pairs implies that $\w_{\tau}(y_i)=\w_{\tau}(y_{i+1})$.
As this holds cyclically, there is $p\in P$ such that $p=\w_{\tau}(y_i)$ for every $i\in [k]$.
However, this implies that $(x_1,y_1),\ldots,(x_k,y_k)$ is an alternating cycle in $I^j_{\tau,p}$, 
which is a contradiction since this set is reversible by assumption. 

Thus, $\bigcup_{p\in P}I^j_{\tau,p}$ is reversible for each $j\in[d]$.  
Since $I_\tau= \bigcup_{j\in [d]}\bigcup_{p\in P}I^j_{\tau,p}$, it follows that $\dim(I_\tau) \leq d$, as desired. 
\end{proof}
Now we can complete the proof of the lemma.
Let $q\in P$ be an element witnessing the maximum value in the right-hand side of the equation in Claim~\ref{claim-max}.
Clearly, $I_{\tau,q}\subset\set{(x,y)\in I: q\leq y\text{ in $P$}}$.
Since 
\[
\dim(I_{\tau,q})= \dim(I_{\tau})\geq\dim(I)/c-2, 
\]
this completes the proof of the lemma.
\end{proof}

We are now ready to prove Theorem~\ref{thm:nowhere-dense-upper-bound}. 
 
\begin{proof}[Proof of Theorem~\ref{thm:nowhere-dense-upper-bound}] 
Let $h \geq 1$ and $t \geq 1$. 
We prove the theorem with the following value for $f(h,t)$:  
\[
 f(h,t):=\binom{m+h}{h}, \quad\text{ where } m:=\binom{t}{2}^{h^t}.
\]
Let thus $P$ be a poset of height at most $h$, let $G$ denote its cover graph, and let $c:=\wcol_{4h-4}(G)$. 
We prove the contrapositive. 
That is, we assume that 
\[
\dim(P)> (3c)^{f(h,t)},
\] 
and our goal is to show that $G$ contains a $\leq\!(2h-3)$-subdivision of $K_t$ as a subgraph. 
For technical reasons, we will need to suppose also that $c > t$. 
This can be assumed without loss of generality, because if not then $\wcol_{3h-3}(G) \leq \wcol_{4h-4}(G) \leq t$, and hence  $\dim(P) \leq 4^{t} \leq (3c)^{f(h,t)}$ by Theorem~\ref{thm:dim-wcol}.\footnote{The reader might object that this makes the proof dependent on Theorem~\ref{thm:dim-wcol}, while we claimed in the introduction that it was not.    
In order to address this perfectly valid point, let us mention that one could  choose instead to add the extra assumption that $c > t$ in the statement of Theorem~\ref{thm:nowhere-dense-upper-bound}; this does not change the fact that it implies the forward direction of Theorem~\ref{thm:main} (in combination with Zhu's theorem). 
However, it seemed rather artificial to do so, since the theorem remains true without this technical assumption.}

\begin{claim}
\label{claim:small-dim}
There exists an antichain $S$ of size $m$ in $P$ such that, letting  $I:=\{(x,y) \in \Inc(P): s \leq y \; \forall s\in S\}$, we have $\dim(I) \geq 2$, and
\[
\dim(\set{(x,y)\in I: q\leq y \text{ in $P$}}) < \dim(I)/3c
\]
for every element $q$ such that there exists $s\in S$ with $s < q$ in $P$.
\end{claim}
\begin{proof}
We define the \emph{height vector} of an antichain $S$ of size at most $m$ in $P$ to be the vector of heights of elements in $S$ ordered in non-increasing order and padded at the end with $0$-entries so that the vector is of size exactly $m$.
Note that $\binom{m+h}{h}$ is the number of size-$m$ vectors with entries in $\{0,1, \dots, h\}$ ordered in non-increasing order.
We enumerate these vectors in lexicographic order with numbers from $0$ to $\binom{m+h}{h}-1 = f(h,t)-1$.
Let $\val(S)$ denote the index of the height vector of $S$ in this enumeration. 
Notice that $\val(\emptyset)=0$. 

To prove the lemma, choose an antichain $S$ with $|S| \leq m$ such that 
\[
\dim(I)>d:=(3c)^{f(h,t)-\val(S)}, 
\]
where $I:=\{(x,y) \in \Inc(P): s \leq y \; \forall s\in S\}$, and with $\val(S)$ maximum.  
Note that $S$ is well defined as $S=\emptyset$ is a candidate. 
Note also that $\dim(I) > d \geq 3c \geq 2$. 
We will show that $S$ and $I$ satisfy the lemma. 

First assume that $q \in P$ is such that there exists $s\in S$ with $s < q$ in $P$ and 
$\dim(I') \geq \dim(I)/3c$, where $I':= \set{(x,y)\in I: q\leq y \text{ in $P$}}$. 
Let $S' := (S - \D[q])\cup\set{q}$. 
Observe that $\val(S') > \val(S)$ and $|S'| \leq |S| \leq m$, since $\D[q] \cap S \neq \emptyset$ and the height of $q$ is strictly larger than the heights of all the elements in $\D[q] \cap S$. 
Moreover, 
\[
 \dim(I') \geq \dim(I)/3c 
 > d/3c 
 = (3c)^{f(h,t)-\val(S)-1}
 \geq (3c)^{f(h,t)-\val(S')},  
\]
showing that $S'$ was a better choice than $S$, a contradiction. 
Hence, there is no such element $q$, and it only remains to show that $|S|=m$. 
Arguing by contradiction, suppose $|S| < m$.  

Consider the poset  $Q:= P - \D[S]$, and let $I_Q := I \cap \Inc(Q)$. 
First, we claim that $I - I_Q$ is reversible in $P$. 
Arguing by contradiction, suppose that this set contains an alternating cycle $(x_1,y_1), \dots, (x_k, y_k)$. 
Since $Q$ is an induced subposet of $P$, for each $i\in [k]$, at least one $x_i$ and $y_i$ must be in $\D[S]$ (otherwise, $(x_i,y_i)$ would be an incomparable pair of $Q$). 
We cannot have $x_i \in \D[S]$, because otherwise $x_i \leq s$ in $P$ for some $s \in S$, and since $s \leq y_i$ in $P$ this would contradict the fact that $x_i$ and $y_i$ are incomparable. 
Thus, $y_i \in \D[S]$. 
However, since $x_{i-1} \leq y_i$ in $P$ (taking indices cyclically), it follows that  $x_{i-1} \in \D[S]$, a contradiction.  
Hence, $I-I_Q$ is reversible, as claimed. 
It follows 
\begin{equation}
\dim(I_Q) \geq \dim(I) - 1. 
\end{equation}
Applying Lemma~\ref{lemma:q-support} on poset $Q$ and set $I_Q$, we obtain an element $q\in Q$ such that $\dim_Q\left( \{(x,y)\in I_Q: q \leq y \text{ in } Q\}\right) \geq \dim_Q(I_Q) / c - 2$. 
(The subscript $Q$ indicates that dimension is computed w.r.t.\ $Q$.) 
Here we use that $Q$ has height at most that of $P$, and thus at most $h$, and that the cover graph $G_Q$ of $Q$ is an (induced) subgraph of $G$ (since $Q$ is an upset of $P$), and thus $\wcol_{4h}(G_Q) \leq \wcol_{4h}(G)\leq c$. 
It only remains to point out that $\dim_Q(I_Q)=\dim(I_Q)$ because $Q$ is an induced subposet of $P$ (that is, a subset of $I_Q$ is an alternating cycle in $Q$ if and only if it is one in $P$), and similarly $\dim_Q\left( \{(x,y)\in I_Q: q \leq y \text{ in } Q\}\right) = \dim\left( \{(x,y)\in I_Q: q \leq y \text{ in } P\}\right)$. 
Putting everything together, we obtain
\begin{align*}
\dim\left( \{(x,y)\in I_Q: q \leq y \text{ in } P\}\right) 
&\geq  \dim(I_Q) / c - 2 \\
&\geq (\dim(I) - 1)/c - 2 \\
&\geq \dim(I)/c - 3. 
\end{align*}

Now, let $S' := (S - \D[q])\cup\set{q}$ and $I':= \set{(x,y)\in I: q\leq y \text{ in $P$}}$. 
Observe that $\val(S') > \val(S)$ and $|S'| \leq |S| +1 \leq m$. 
Moreover, 
\begin{align*}
 \dim(I') &\geq \dim(\set{(x,y)\in I_Q : q\leq y \text{ in $P$}}) \\
 &\geq \dim(I)/c - 3 \\
 &> (3c)^{f(h,t)-\val(S)}/c-3\\
 &\geq 3\cdot (3c)^{f(h,t)-\val(S)-1}-3\\
 &\geq (3c)^{f(h,t)-\val(S')}. 
\end{align*}
(For the last inequality we use that $\val(S) < \val(S') < f(h,t)$.)
This shows that $S'$ is a better choice than $S$, a contradiction. 
\end{proof}

Let $S$ denote an antichain given by Claim~\ref{claim:small-dim}, and let $I$ denote the corresponding set of incomparable pairs. 
The next claim will be used to build the desired subdivision of $K_t$, the set $V$ will be the set of branch vertices. 

\begin{claim} 
There exist disjoint sets $V\subset P$ and $R\subset S$ such that 
\begin{enumerate2}
  \item\label{inv:V-and-R-sizes} $\norm{V}=t$ and $|R|\geq m^{h^{-t}} = \binom{t}{2}$,
  \item\label{inv:clean-branching-for-V} for all $(v,r)\in V \times R$, there is $p\in P$ such that $v$ covers $p$ in $P$ and $\D[p]\cap R=\set{r}$.
 \end{enumerate2}
\end{claim}
\begin{proof}
Choose disjoint sets $V\subset P$ and $R\subset S$ satisfying \ref{inv:clean-branching-for-V} and 
\[ 
\norm{V}=j \quad \textrm{ and } \quad |R|\geq m^{h^{-j}} = \binom{t}{2}^{h^{t-j}}, 
\]
 with $j \leq t$ as large as possible. 
 Note that $V=\emptyset$ and $R=S$ is a candidate, hence this choice is possible. 
 We claim that $j=t$, which implies the lemma. 
 Arguing by contradiction, assume $j<t$. 

 For every $v\in V$, there is $r\in R\subseteq S$ such that $r<v$ in $P$ by~\ref{inv:clean-branching-for-V}.
Hence, by Claim~\ref{claim:small-dim} 
\begin{align*}
\dim\left(\left\{(x,y)\in I: y \not\in \Up[V]\right\} \right)
&\geq \dim(I) - \sum_{v\in V} \dim(\set{(x,y)\in I: v\leq y \text{ in $P$}})\\
&> \dim(I) - \norm{V}\cdot \dim(I)/(3c)\\
&\geq 2 (1- t/(3c))\\
&> 1.
\end{align*}
(This is the place in the proof where we use our assumption that $c > t$.) 
It follows that the left-hand side is at least $2$.
Therefore, $\set{y: (x,y)\in I} - \Up[V]$ is not empty. 
(Recall that the dimension of an empty set of incomparable pairs is $1$.)
Choose some element $y$ in this set.

Now, starting from the element $y$, we go down along cover relations in the poset $P$. 
Initially we set $v:=y$, and as long as there is a element $x\in P$ such that $v$ covers $x$ in $P$ and 
\begin{equation}
\norm{\D[x]\cap R} > \norm{\D[v]\cap R}/m^{h^{-(j+1)}},\quad \textrm{we update $v:=x$.}\label{eq:v-update}
\end{equation}

Note that the process must stop as the height of $v$ is decreasing in every move. 
We claim that $v$ never goes down to a minimal element nor to an element in $R$.
Indeed, if in the above procedure we are considering an element $x$ covered by $v$, then after at most $h-2$ steps we are done, and hence 
\[
\norm{\D[v]\cap R} 
> \frac{\norm{\D[y]\cap R}}{m^{(h-2)h^{-(j+1)}}} 
= \frac{\norm{R}}{m^{(h-2)h^{-(j+1)}} }
\geq m^{h^{-j} - (h-2)h^{-(j+1)}} 
\geq m^{h^{-(j+1)}}.
\]
(Note that $\norm{\D[y]\cap R}=\norm{R}$ since $s \leq y$ in $P$ for all $s\in S$, by our choice of $y$.) 
Now, if $x$ is a minimal element or $x\in R$, then $\norm{\D[x]\cap R} \leq 1$ (note that $\norm{\D[x]\cap R} = 1$ when $x\in R$ because $R\subseteq S$ and $S$ is an antichain).  
Hence, the inequality of \eqref{eq:v-update} cannot hold strictly.  
Therefore, at the end of the process $v$ is not a minimal element of $P$ nor is included in $R$, as claimed.  

Consider now the set $Z$ consisting of all elements that are covered by $v$ in $P$.
Since $v\not\in R$, we have $\D[v]\cap R \subseteq \D[Z]$.
Let $Z'$ be an inclusion-wise minimal subset of $Z$ such that $\D[v]\cap R \subseteq \D[Z']$.
The minimality of $Z'$ allows us to fix for every $z\in Z'$ an element $r_z\in \D[v]\cap R$ such that $r_z\in \D[z]$ and $r_z\not\in \D[z']$ for every $z'\in Z' - \{z\}$. 
Let 
\[
 V' := V\cup\set{v}\quad \textrm{ and }\quad R' := \set{r_z: z\in Z'}. 
\]
Recall that we have chosen $y$ such that $w\not\leq y$ in $P$ for every $w\in V$.
On the other hand, by our procedure we have $v \leq y$ in $P$. 
Thus $v\not\in V$, and $\norm{V'}=j+1$.
We will show that the pair $(V', R')$ was a candidate for our choice at the beginning of the proof, and hence a better choice than the pair $(V, R)$. 

Since elements $r_z$, $r_{z'}$ are distinct for distinct $z,z'\in Z'$, we have $\norm{R'}=\norm{Z'}$.
Moreover,
\begin{align*}
 \norm{\D[v]\cap R} = \norm{\D[Z']\cap R} 
 \leq \sum_{z\in Z'}\norm{\D[z]\cap R} 
 \leq \norm{Z'}\cdot \frac{\norm{\D[v]\cap R}}{m^{h^{-(j+1)}}}.
\end{align*} 
We deduce that
\[
 \norm{R'}=\norm{Z'}\geq  m^{h^{-(j+1)}}.
\]

Note that $V'$ and $R'$ are disjoint since $V$ is disjoint from $R$ and since $v$ is not contained in $R$, and thus not in $R'$ either.

It remains to verify~\ref{inv:clean-branching-for-V} for $(V',R')$.
Since $R'\subseteq R$, we only need to check this property for the new vertex $v$.
Consider an element $r\in R'$.
By the definition of $R'$ there is $z\in Z'$ such that $r=r_z$.
Recalling the way we defined $r_z$, we obtain $\D[z]\cap R'=\set{r_z}$.
This shows that the pair $(V', R')$ was a better choice than $(V, R)$, a contradiction.  
\end{proof}

For the rest of the proof let $(V,R)$ be a pair given by the above claim. 
As mentioned, the vertices of $V$ will serve as the branch vertices of the subdivision of $K_t$ we are building. 
Next, we connect these vertices pairwise with internally vertex-disjoint paths. 
By \ref{inv:V-and-R-sizes} we have
\[
 \norm{V}=t \quad\textrm{ and }\quad \norm{R} \geq \binom{t}{2}.
\]
Thus, for each unordered pair $\{v_1,v_2\} \subseteq V$ we can choose a corresponding element $r_{\{v_1,v_2\}}\in R$ in such a way that all chosen elements are distinct. 
Furthermore, by Invariant~\ref{inv:clean-branching-for-V}, there are elements $p_1$ and $p_2$ covered respectively by $v_1$ and $v_2$ in $P$ such that $\D[p_1]\cap R=\D[p_2]\cap R=\set{r_{\{v_1,v_2\}}}$.  
Let 
\[
 r_{\{v_1,v_2\}}=u_1<u_2<\cdots<u_{k}<p_{1}<v_1\ \text{ and }\ r_{\{v_1,v_2\}}=w_1<w_2<\cdots<w_{\ell}<p_{2}<v_2
\]
be covering chains in $P$.
Clearly, the union of the two covering chains contains a path connecting the vertices $v_1$ and $v_2$ in $G$; fix such a path $Q_{\{v_1,v_2\}}$ for the unordered pair $\{v_1,v_2\}$.  
Observe that the path $Q_{\{v_1,v_2\}}$ has length at most $2h-2$. 

Connecting all other pairs of vertices in $V$ in a similar way, we claim that the union of these paths forms a subdivision of $K_t$. 
All we need to prove is that whenever there is $z\in Q_{\{v_1,v_2\}}\cap Q_{\{v_1',v_2'\}}$ for distinct sets $\set{v_1,v_2},\set{v_1',v_2'}\subseteq V$, 
then $z$ is an endpoint of both paths.
Suppose to the contrary that $z$ is an internal vertex of one path, say of $Q_{\{v_1,v_2\}}$.
By our construction, there are elements $p_1$ and $p_2$ covered respectively by $v_1$ and $v_2$ in $P$ with $\D[p_1]\cap R=\D[p_2]\cap R=\set{r_{\{v_1,v_2\}}}$.
Furthermore, we have $z\leq p_1$ or $z\leq p_2$ in $P$.
Say $z\leq p_1$ without loss of generality. 
From $z\in Q_{\{v_1',v_2'\}}$ we deduce that $r_{\{v_1',v_2'\}}\leq z$ in $P$, which implies that $r_{\{v_1',v_2'\}}\leq p_1$ in $P$.
However, it follows that $r_{\{v_1',v_2'\}}\in \D[p_1]\cap R=\set{r_{\{v_1,v_2\}}}$ and hence $r_{\{v_1,v_2\}}=r_{\{v_1',v_2'\}}$, a contradiction to our construction.
We conclude that both paths $Q_{\{v_1,v_2\}}$ and $Q_{\{v_1',v_2'\}}$ are indeed internally disjoint.

Finally, since all the paths $Q_{\{v_1,v_2\}}$ ($\set{v_1,v_2}\subseteq V$) have length at most $2h-2$, this shows the existence of a $\leq\!(2h-3)$-subdivision of $K_t$ in $G$, as desired.  
This completes the proof of Theorem~\ref{thm:nowhere-dense-upper-bound}.
\end{proof}

\section{Applications}
\label{sec:applications}

In this section we discuss some applications of Theorem~\ref{thm:dim-wcol}, starting with posets whose cover graphs have bounded genus. 
It was shown by van den Heuvel, Ossona de Mendez, Quiroz, Rabinovich, and Siebertz~\cite{HOQRS} that 
\[
 \wcol_r(G)\leq \left(2g+\binom{r+2}{2}\right)\cdot (2r+1)
\]
for every graph $G$ with genus $g$. 
Combining this inequality with Theorem~\ref{thm:dim-wcol}, we obtain the following upper bound.
\begin{corollary}
 For every poset $P$ of height at most $h$ whose cover graph has genus $g$, 
 \[
  \dim(P)\leq 4^{\left(2g+\binom{3h-1}{2}\right)\cdot(6h-5)}.
 \]
\end{corollary} 
For fixed genus, this is a $2^{\mathcal{O}(h^3)}$ bound on the dimension. 
In particular, this improves on the previous best bound for posets with planar cover graphs~\cite{JMMTWW}, which was doubly exponential in the height. 

It is in fact suspected that posets with planar cover graphs have dimension at most linear in their height. 
This was recently proved~\cite{JMW_PlanarPosets} for posets whose {\em diagrams} can be drawn in a planar way; these posets form a strict subclass of posets with planar cover graphs. 
Regarding posets with planar cover graphs, Kozik, Micek, and Trotter recently announced that they could prove a polynomial bound on the dimension. 
Let us remark that it is rather remarkable that linear or polynomial bounds can be obtained when assuming that the poset has a planar diagram or a planar cover graph, respectively.  
Indeed, for the slightly larger class of posets with $K_5$-minor-free cover graphs, constructions show that the dimension can already be exponential in the height, as shown in~\cite{JMW_PlanarPosets}. 
(This also follows from Theorem~\ref{thm:lb-treewidth} below, applied with $t=3$.) 

We continue our discussion with graphs of bounded treewidth. 
Let us first quickly recall the definitions of tree decompositions and treewidth (see e.g.\ Diestel~\cite{Diestel} for an introduction to this topic). 
A {\em tree decomposition} of a graph $G$ is a pair consisting of a tree $T$ and a collection $\{B_x \subseteq V(G) : x \in V(T)\}$ of sets of vertices of $G$ called {\em bags}, one for each node of $T$, satisfying: 
\begin{itemize}
	\item each vertex $v\in V(G)$ is contained in at least one bag; 
	\item for each edge $uv\in E(G)$, there is a bag containing both $u$ and $v$, and
	\item for each vertex $v\in V(G)$, the set of nodes $x\in V(T)$ such that $v \in B_x$ induces a subtree of $T$. 
\end{itemize}
The {\em width} of the tree decomposition is $\max_{x\in V(T)} \norm{B_x}-1$. 
The {\em treewidth} of $G$ is the minimum width of a tree decomposition of $G$. 

Grohe, Kreutzer, Rabinovich, Siebertz, and Stavropoulos~\cite{GKRSS} showed that 
\begin{align*}
 \wcol_r(G)\leq \binom{t+r}{t} 
\end{align*}
for every graph $G$ of treewidth $t$.  
Combining Theorem~\ref{thm:dim-wcol} with the above bound, we obtain a single exponential bound.
\begin{corollary}
 For every poset $P$ of height at most $h$ with a cover graph of treewidth $t$, 
 \[
  \dim(P)\leq 4^{\binom{t+3h-3}{t}}.
 \]
\end{corollary}
For fixed $t$, this is a $2^{\mathcal{O}(h^t)}$ bound on the dimension, which improves on the doubly exponentional bound in~\cite{JMMTWW}. 
Surprisingly, this upper bound turns out to be essentially best possible:

\begin{theorem}\label{thm:lb-treewidth}
 Let $t\geq 3$ be fixed.
 For each $h\geq 4$, there exists a poset $P$ of height at most $h$ whose cover graph has treewidth at most $t$, and such that
 \[
  \dim(P)\geq 2^{\Omega(h^{\lfloor(t-1)/2\rfloor})}.
 \]
\end{theorem}

This theorem will be implied by the following slightly more technical theorem, which is an extension of the construction for treewidth $3$ in~\cite{JMW_PlanarPosets}.  
In this theorem, we use the following terminology: If a poset $P$ is such that the sets of its minimal and maximal elements induce a standard example $S_k$ then we call {\em vertical pair} each of the $k$ pairs $(a,b)$ with $a$ a minimal element, $b$ a maximal element, and $(a, b) \in \Inc(P)$. 

\begin{theorem}\label{thm:lb-tw-construct}
 For every $h\geq 1 $ and $t\geq 1$, there exists a poset $P_{h,t}$ and a tree decomposition of its cover graph such that
 \begin{enumerate}
  \item\label{item:tw-height} $P_{h,t}$ has height $2h$;
  \item\label{item:tw-standard} the minimal and maximal elements of $P_{h,t}$ induce the standard example $S_k$ with $k=2^{\binom{h+t-1}{t}}$; 
  \item\label{item:tw-width} the tree decomposition has width at most $2t+1$, and
  \item\label{item:tw-bag} for each vertical pair $(a,b)$ in $P_{h,t}$ there is a bag of the tree decomposition containing both $a$ and $b$.
 \end{enumerate}
\end{theorem}
\begin{proof}
 We prove the theorem by induction on $h$ and $t$.
 Let us first deal with the case $h=1$, which serves as the base cases for the induction.
 If $h=1$, then it is easy to see that letting $P_{1,t}$ be the standard example $S_2$ fulfills the desired conditions. 
 For the tree decomposition, it suffices to take a tree consisting of a single node whose bag contains all four vertices. 
 (We note that we could in fact take $P_{1,t} := S_{2t+1}$ and increase slightly the bound in~\ref{item:tw-standard} but the gain is negligible.)
 
 Next, for the inductive case, suppose that $h \geq 2$.  
 We treat separately the cases $t=1$ and $t \geq 2$. 

 First, suppose that $t=1$.
 The poset $P_{h,1}$ is defined using the inductive construction illustrated in Figure~\ref{fig:height-tw-ex} (left):
 We start with the poset $P_{h-1,1}$, and for each vertical pair $(a,b)$ in $P_{h-1,1}$ we introduce four elements $x_1,x_2,y_1,y_2$ forming a standard example $S_2$ (with $x$'s and $y$'s being minimal and maximal elements, respectively). 
 Then we add the relations $x_1 < a$ and $x_2 < a$, and $b < y_1$ and $b < y_2$, and take the transitive closure. 
 This defines the poset $P_{h,1}$.   
 It is easy to see that the height of $P_{h,1}$ is exactly the height of $P_{h-1,1}$ plus $2$, which implies \ref{item:tw-height}.
 It is also easily checked that the number of minimal (maximal) elements in $P_{h-1,1}$ is twice the number in $P_{h-1,1}$, which was $2^{h-1}$, and that the union of minimal and maximal elements induce the standard example $S_{2^h}$, showing~\ref{item:tw-standard}.
 
 Now, consider a tree decomposition of the cover graph of $P_{h-1,1}$ satisfying \ref{item:tw-width} and \ref{item:tw-bag}.
 For each vertical pair $(a,b)$ of $P_{h-1,1}$, consider a node $z$ of the tree whose bag $B_z$ contains both $a$ and $b$, and let $x_1,x_2,y_1,y_2$ be the four elements introduced when considering $(a,b)$ in the definition of $P_{h,1}$. 
 Extend the tree decomposition by adding three new nodes $z', z'_1, z'_2$ with bags $B_{z'}:= \{a,b,y_1,y_2\}$, $B_{z'_1}:= \{a,x_1,y_1,y_2\}$, $B_{z'_2}:= \{a,x_2,y_1,y_2\}$, and adding the three edges $zz', z'z'_1, z'z'_2$ to the tree, as illustrated in Figure~\ref{fig:tw-3-ex}. 
 Clearly, once this extension is done for each vertical pair of $P_{h-1,1}$, the resulting tree decomposition of the cover graph of $P_{h,1}$ satisfies \ref{item:tw-width} and \ref{item:tw-bag}. 
 
 Next, suppose that $t\geq 2$.
 We start with a copy of $P_{h-1,t}$ and let $(a_1,b_1),\ldots,(a_\ell,b_\ell)$ denote its vertical pairs.
 For each vertical pair $(a_i,b_i)$, we introduce a copy $P^i$ of $P_{h,t-1}$, and add the relation $x < a_i$ for each minimal element $x$ of $P_i$, and the relation $b_i < y$ for each maximal element $y$ of $P_i$; see Figure~\ref{fig:height-tw-ex} (right). 
 Then $P_{h,t}$ is obtained by taking the transitive closure of this construction.
 
 Observe that $P_{h,t}$ has height $2h$, thus~\ref{item:tw-height} holds. 
 Moreover, the minimal and maximal elements of $P_{h,t}$ induce a standard example $S_k$ with  
 \[
  k=2^{\binom{h+t-2}{t}}\cdot 2^{\binom{h+t-2}{t-1}}=2^{\binom{h+t-1}{t}} 
 \]
 by the induction hypothesis, showing~\ref{item:tw-standard}. 

 Next, consider the tree decomposition of the cover graph of $P_{h-1,t}$ given by the induction hypothesis.
 We extend this tree decomposition by doing the following for each vertical pair $(a_i,b_i)$ of $P_{h-1,t}$: 
 Consider a node $z^i$ of the tree whose bag contains both $a_i$ and $b_i$. 
 Take the tree decomposition of the cover graph of $P^i$ given by the induction hypothesis and denote its tree by $T^i$ (on a new set of nodes).  
 Then add an edge between the node $z^i$ and an arbitrary node of $T^i$, and  add $a_i$ and $b_i$ to every bag of nodes coming from $T^i$.
 It is easily checked that this defines a tree decomposition of the cover graph of $P_{h,t}$, of width at most $2t+1$, such that for each vertical pair $(a,b)$ of $P_{h,t}$ there is bag containing both $a$ and $b$. 
 Therefore, properties \ref{item:tw-width} and \ref{item:tw-bag} are satisfied. 
\end{proof}

\begin{figure}[t]
 \centering
 \includegraphics{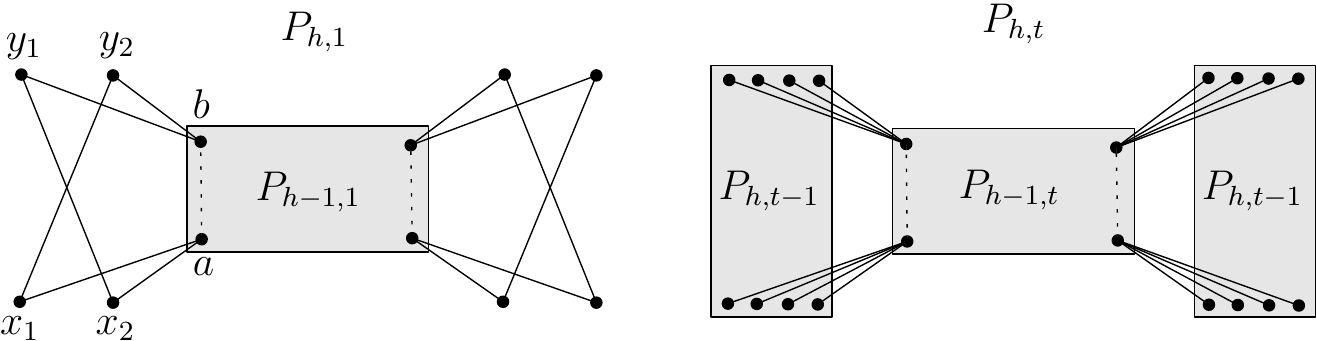}
 \caption{Inductive construction of $P_{h,t}$.}
 \label{fig:height-tw-ex}
\end{figure}

\begin{figure}[h]
 \centering
 \includegraphics{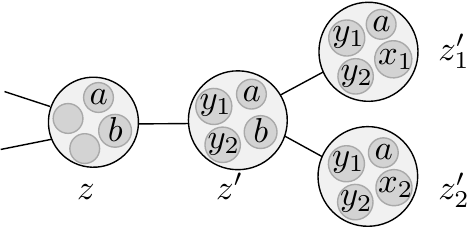}
 \caption{Extending the tree decomposition.}
 \label{fig:tw-3-ex}
\end{figure}

\begin{proof}[Proof of Theorem~\ref{thm:lb-treewidth}]
 Let $t\geq 3$ and $h\geq 4$. 
 Then we set $h':=\lfloor h/2\rfloor$ and $t'=\lfloor(t-1)/2\rfloor$.
 With these values, the poset $P_{h',t'}$ from Theorem~\ref{thm:lb-tw-construct} has height $2h'\leq h$ and its cover graph has treewidth at most $2t'+1\leq t$.
 Moreover, 
 \[
  \dim(P_{h',t'})\geq 2^{\binom{h'+t'-1}{t'}}=2^{\Omega(h'^{t'})}=2^{\Omega(h^{\lfloor (t-1)/2\rfloor})}.
 \]
(Recall that the asymptotics in the theorem statement are taken with respect to $h$ with $t$ being a fixed constant.)
\end{proof}

We pursue with the case of posets whose cover graphs exclude $K_t$ as a minor.
It was shown by van den Heuvel {\it et al.}~\cite{HOQRS} that 
 \[
  \wcol_r(G)\leq \binom{r+t-2}{t-2}\cdot (t-3)(2r+1)\in\mathcal{O}(r^{t-1})
 \]
for every graph $G$ excluding $K_t$ as a minor. 
Together with Theorem~\ref{thm:dim-wcol}, this yields the following improvement on the previous best bound~\cite{MW15}, which was doubly exponential in the height (for fixed $t$).
\begin{corollary}
 For every poset $P$ of height at most $h$ whose cover graph excludes $K_t$ as a minor, 
 \[
  \dim(P)\leq 4^{\binom{3h+t-5}{t-2}\cdot (t-3)(6h-5)}. 
 \]
\end{corollary}
For a fixed integer $t \geq 5$, this $2^{\mathcal{O}(h^{t-1})}$ bound is again essentially best possible by Theorem~\ref{thm:lb-treewidth} (using an upper bound of $t-2 \geq 3$ on the treewidth), because graphs of treewidth at most $t-2$ cannot contain $K_{t}$ as a minor. 
On the other hand, it is no coincidence that we cannot use Theorem~\ref{thm:lb-treewidth} in this way when $t \leq 4$: Indeed, posets whose cover graphs exclude $K_4$ as a minor (or equivalently, have treewidth at most $2$) have dimension bounded by a universal constant (at most $1276$), irrespectively of their height~\cite{JMTWW}.

Regarding graphs $G$ that exclude $K_t$ as a topological minor, it is implicitly proven in the work of Kreutzer, Pilipczuk, Rabinovich, and Siebertz~\cite{KPRS} that these graphs satisfy
\[
 \wcol_r(G)\leq 2^{\mathcal{O}(r\log r)} 
\]
when $t$ is fixed. 
Combining this inequality with Theorem~\ref{thm:dim-wcol} we get a slight improvement upon the bound derived in~\cite{MW15}, however the resulting bound remains doubly exponential: 
\begin{corollary}
 Let $t \geq 1$ be a fixed integer. Then, every poset $P$ of height at most $h$ whose cover graph excludes $K_t$ as a topological minor satisfies  
 \[
  \dim(P)\leq 2^{2^{\mathcal{O}(h\log h)}}.
 \]
\end{corollary}

See Figure~\ref{fig:hierarchy} for a summary of the best known upper bounds and extremal examples for the various graph classes discussed in this section (and a few more).  
Bounds not already mentioned in the text can be found in~\cite{TM77,FTW13,JMW_PlanarPosets,Veit_PhD}. 

\begin{figure}[ht!]
 \centering
 \includegraphics{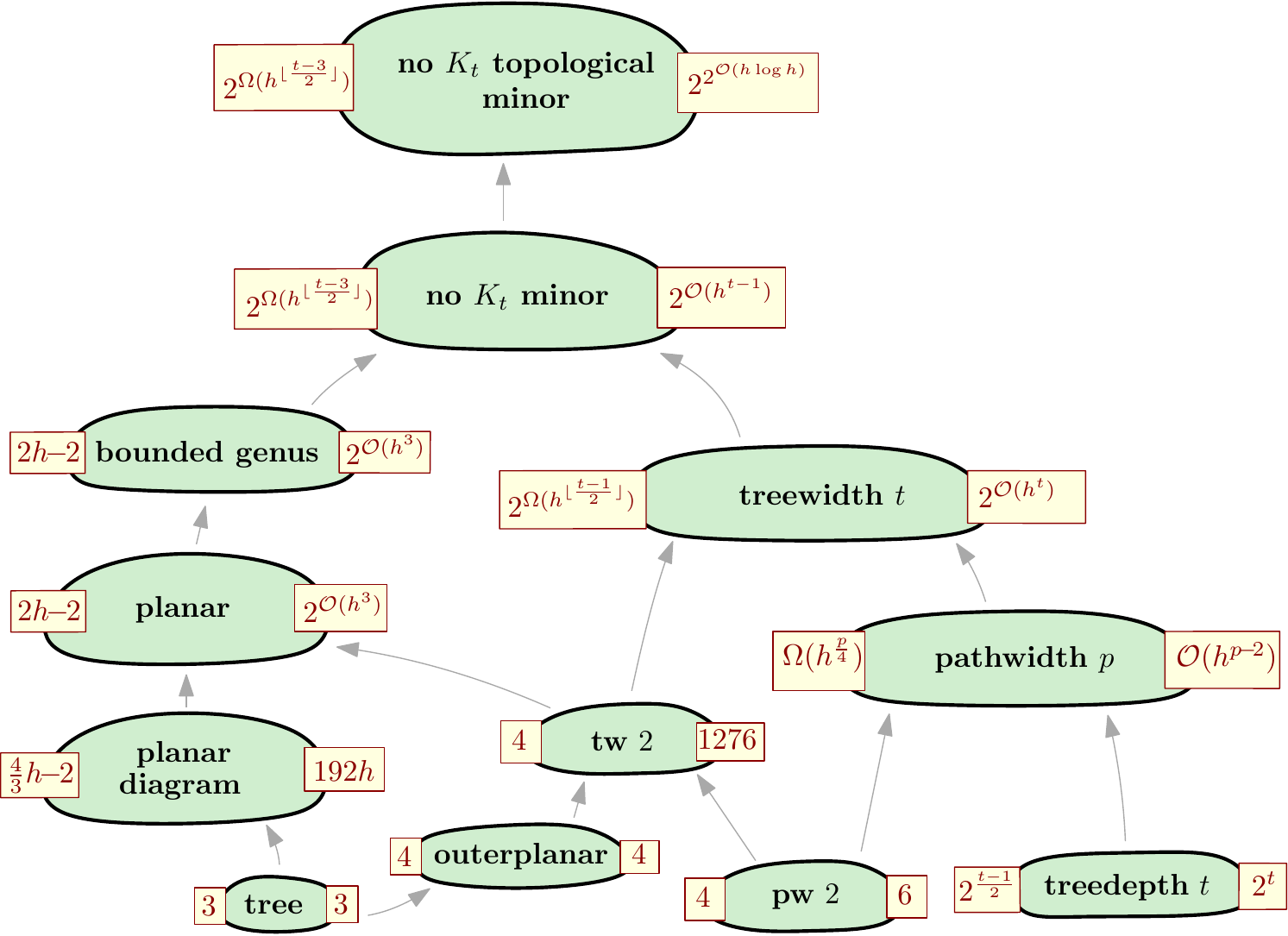}
 \caption{Summary of known bounds.\label{fig:hierarchy}} 
\end{figure}

\section{Open problems}
\label{sec:open_problems}

One remaining open problem is to prove the backward direction of Conjecture~\ref{conj:bounded_exp}, which we restate here: 

\begin{conjecture} 
Let $\calC$ be a monotone class of graphs such that for every fixed $h\geq 1$, posets of height at most $h$ whose cover graphs are in $\calC$ have bounded dimension. 
Then  $\calC$ has bounded expansion. 
\end{conjecture}

As a first step, one could try to show that graphs in the class  $\calC$ have bounded average degree. 
In this direction, we offer the following related conjecture.

\begin{conjecture} 
Let $\calC$ be a monotone class of bipartite graphs such that, seeing the graphs in $\calC$ as posets of height (at most) $2$, these posets have bounded dimension. 
Then  the graphs in $\calC$ have bounded average degree. 
\end{conjecture}

\section*{Acknowledgements} 

We are much grateful to the anonymous referees for their very helpful comments. 
In particular, we thank one referee for pointing out an error in the proof of Claim~\ref{claim:small-dim} regarding how element $q$ was chosen, and another referee for her/his many suggestions on how to improve the exposition of the proofs and shorthen the arguments. 


\bibliographystyle{plain}
\bibliography{posets-dimension}

\end{document}